\documentclass[11pt]{article}

\usepackage{amsthm,amssymb,enumitem,mathtools}
\usepackage{subcaption}
\usepackage{scalerel}
\usepackage{times}

\usepackage[margin=1in,letterpaper,portrait]{geometry}

\usepackage[latin1]{inputenc}
 \usepackage{color}

\usepackage{scrextend}

\theoremstyle{plain}
\theoremstyle{definition}
\newtheorem{theorem}{Theorem}[section]

\newtheorem{remark}[theorem]{Remark}

\newtheorem{example}[theorem]{Example}
\newtheorem{proposition}[theorem]{Proposition}
\newtheorem{corollary}[theorem]{Corollary}

\DeclareMathAlphabet{\mathpzc}{OT1}{pzc}{m}{it}

\usepackage{algorithm}
\usepackage{algorithmic}

\begin{document}

\begin{center}
 {\large\sc Raney numbers, threshold sequences and Motzkin-like paths}
 \bigskip
 
 {\bf Irena Rusu\footnote{Irena.Rusu@univ-nantes.fr}}
 
 {LS2N, UMR 6004, Universit\'e de Nantes, France}

\end{center}
\bigskip

%
%
%
%
%
%
%
%
%
 \begin{abstract}%
We provide new interpretations for a subset of Raney numbers, involving threshold sequences and 
Motzkin-like paths with long up and down steps.

Given three integers $n,k,l$ such that $n\geq 1, k\geq 2$ and $0\leq l\leq k-2$, a $(k,l)$-threshold sequence of length $n$  is
any strictly increasing sequence $S=(s_1s_2\ldots s_n)$ of integers such that $ki\leq s_i\leq kn+l$. 
These sequences  are in bijection with ordered $(l+1)$-tuples of $k$-ary trees. We prove this result and
identify the Raney numbers that count the $(k,l)$-threshold sequences. As a consequence, when $k=2$ and $k=3$, we
deduce combinatorial identities involving Catalan numbers and powers of 2, and
respectively Fuss-Catalan and Raney numbers. Finally, we show how to represent threshold sequences as Motzkin-like paths with long
up and down steps, and deduce that these paths are enumerated by the same Raney numbers. 

 \end{abstract}
%


\section{Introduction}

In this paper, a $k$-ary tree ($k\geq 2$) is any tree whose internal nodes, the root included, have exactly $k$ children.  
It follows that such a tree either is a unique node (the tree is then called {\em trivial}), or is made of an internal node 
(the root) and $k$ smaller $k$-ary subtrees.  We assume  that the $k$ subtrees are ordered from left to right, meaning that 
two trees are equal if and only if either both of them are trivial, or both of them are non-trivial and they have the 
same $k$-subtrees identically ordered.

With the notation $C^{(k)}_n$ for the number of $k$-ary trees with $n$ internal nodes, we then have $C^{(k)}_0=1$ and:

\begin{equation}
C^{(k)}_n=\sum_{\substack{j_1+j_2+\ldots +j_k=n-1\\ \forall\, h\,:\,\, 0\leq j_h\leq n-1}}^{}C^{(k)}_{j_1}C^{(k)}_{j_2}\ldots C^{(k)}_{j_k}
\label{eq:Tnrec}
\end{equation}

This recurrence defines the so-called Fuss-Catalan numbers,  for which the following closed form is known \cite{graham1994concrete}:

\begin{equation}
C^{(k)}_n=\frac{1}{(k-1)n+1}\binom{kn}{n}.
\label{eq:generalizedCatalan}
\end{equation}

The Fuss-Catalan numbers are the particular case with $r=1$ of the Raney numbers, defined as follows \cite{raney1960functional}:
\bigskip

\begin{equation}
R_n^{(k,r)}=\frac{r}{kn+r}\binom{kn+r}{n}=\frac{r}{(k-1)n+r}\binom{kn+r-1}{n}
 \label{eq:Raney}
\end{equation}

As proved in \cite{hilton1991catalan}, the Raney numbers are also related to the Fuss-Catalan numbers by the following relation:

\begin{equation}
R_n^{(k,r)}=\sum_{\substack{i_1+i_2+\ldots +i_r=n\\ \forall\, h\,:\,\, 0\leq i_h\leq n}}^{}C^{(k)}_{i_1}C^{(k)}_{i_2}\ldots C^{(k)}_{i_r}
\label{eq:RaneyRecCatalan}
\end{equation}

Using the interpretation of each $C^{(k)}_{i_h}$, which represents the number of $k$-ary trees with $i_h$ internal nodes,
the previous recurrence relation implies that:

\begin{remark}
The Raney number $R_n^{(k,r)}$ counts the number of ordered $r$-tuples of  $k$-ary trees with a total number
of $n$ internal nodes.
\label{rem:Raneycountsktrees}
\end{remark}
\bigskip


The cases where $k=2$ and $k=3$ are particularly useful. Binary trees with $n$ internal nodes, obtained when $k=2$, are counted by  
$C^{(2)}_n=\frac{1}{n+1}\binom{2n}{n}$, for $n\geq 0$, also called {\em Catalan number} and denoted $C_n$.
Ternary trees with $n$ internal nodes, obtained when $k=3$,  are counted by the Fuss-Catalan numbers 
$C^{(3)}_n=\frac{1}{2n+1}\binom{3n}{n}=R_n^{(3,1)}$, and  the ordered pairs of ternary trees with total number of $n$ internal nodes 
are counted by the Raney number $R_n^{(3,2)}=\frac{2}{2n+2}\binom{3n+1}{n}=\frac{1}{n+1}\binom{3n+1}{n}$.
To simplify the notations, we denote by $T_n=C^{(3)}_n=\frac{1}{2n+1}\binom{3n}{n}$ the number of ternary trees
and by $U_n=R_n^{(3,2)}=\frac{1}{n+1}\binom{3n+1}{n}$ the number of ordered pairs of ternary trees, both with $n$ internal nodes. In the Online Encyclopedia of Integer Sequences \cite{oeis}, the sequences $T_n$ and $U_n$ 
are respectively OEIS A001764 and OEIS A006013, and enumerate many other combinatorial objects.

A {\em $(k,l)$-threshold sequence of length} $n$, for $n\geq 1$, $k\geq 2$ and $0\leq l\leq k-2$, is any strictly increasing sequence 
$S=(s_1s_2\ldots s_n)$ of integers such that $ki\leq s_i\leq kn+l$. See Figure \ref{fig:standard}. (Note that the case $l=k-1$ is not significant, since 
the resulting sequences are easily identified as the $n$-length prefixes of the $(k,0)$-threshold sequences of length $n+1$.) The contiguous subsequence $(s_p s_{p+1} \ldots s_q)$
of $S$ is denoted by $S[p,q]$. A {\em proper} $(k,l)$-threshold sequence
is a $(k,l)$-threshold sequence such that $s_n=kn+l$, {\em i.e.} is either a $(k,0)$-sequence, or a $(k,l)$-threshold sequence with $l\geq 1$ 
which is not a $(k,l-1)$-threshold sequence. Threshold-like sequences appear for instance in the characterization of maximally admissible pinnacle sets \cite{rusu2021admissible} (with $k=3$).

\begin{example}
Let $k=3$ and $n=6$. Then $S_1=(3\, 6\, 14\, 15\, 17\, 18)$ is a proper $(3,0)$-threshold sequence, whereas $S_2=(3\, 6\, 14\, 15\, 17\, 19)$
is a proper $(3,1)$-threshold sequence of length $6$. Moreover, $S_1$ is also a $(3,1)$-threshold sequence, but not a proper one. 
The sequence $S_3=(3\, 4\, 14\, 15\, 17\, 18)$ is not a $(3,0)$-threshold sequence since $s_2<6$.
\end{example}

\begin{remark}
The sequences $S'=(s'_1s'_2\ldots s'_n)$ of integers such that $ki+d\leq s'_i\leq kn+l+d$, for a fixed integer $d$, are in 
bijection with the $(k,l)$-threshold sequences, and are thus counted by the same formulas. 
We call them {\em $(k,l)$-sequences with offset $d$}.
 \label{rem:decalage}
\end{remark}

In this paper, we show that $(k,l)$-threshold sequences  are related to $k$-ary trees in the sense that
$(k,l)$-threshold sequences of length $n$ are in bijection with ordered $(l+1)$-tuples of $k$-ary trees with a total of $n$ internal nodes, and are
thus counted by the Raney numbers $R_n^{(k,l+1)}$. We further deduce combinatorial identities involving Catalan, Fuss-Catalan and Raney numbers
as well as bijections between $(k,l)$-threshold sequences and related combinatorial objects.

To this end, in Section \ref{sec:three} we first consider the particular case $k=3$, and obtain an implicit bijection by showing 
that the number of $(3,l)$-threshold sequences of length $n$
equals that of $(l+1)$-tuples of ternary trees with $n$ internal nodes, for $l=0,1$. Then, in Section \ref{sec:explainbijections} we propose, 
for arbitrary values $k$ and $l$, an explicit
bijection between $(k,l)$-threshold sequences and $(l+1)$-tuples of $k$-ary trees.  In Section \ref{sec:combinatorailidentities} we present 
several combinatorial identities resulting
from the recurrence relations proved in Sections \ref{sec:three} and \ref{sec:explainbijections}. Finally, in Section \ref{sec:other},
we represent threshold sequences as Motzkin-like paths with up steps $(1,u)$, $u\geq 1$, horizontal
steps $(1,0)$ and down steps $(1,-d)$, $1\leq d\leq k-1$. This representation allows us to deduce counting formulas for these
paths too. Section \ref{sec:conclusion} is the conclusion.

\section{$(3,l)$-Threshold sequences of length $n$}\label{sec:three}

In this section, $k=3$ and the $(3,l)$-threshold sequences are called {\em simple 3-threshold sequences} when $l=0$, and {\em double
3-threshold sequences} when $l=1$.
Simple $3$-threshold sequences are thus also double $3$-threshold sequences. 

\begin{figure}[t]
 \centering
 \vspace*{-3cm}
 \includegraphics[width=0.95\textwidth]{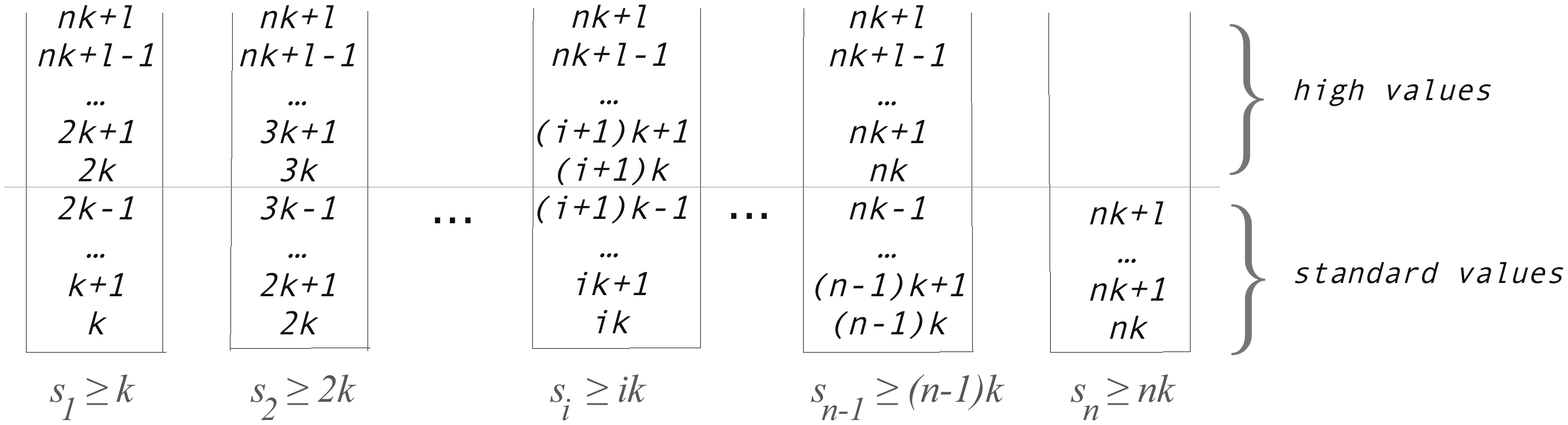}
 \vspace*{-3cm}
 \caption{Construction of a $(k,l)$-threshold sequence. For each $s_i$, a standard or a  high value must be chosen, larger than $s_{i-1}$.} 
 
 \label{fig:standard}
\end{figure}

Let $a_n$ ($b_n$) be the number of simple (double) $3$-threshold sequences of length $n$. Note that $a_1=1, a_2=3$, whereas 
$b_1=2, b_2=7$. By convention, we define $a_0=b_0=1$.
We use various ways to count $a_n$ and $b_n$, then combine them to show that $a_n$ satisfies the recurrence relation in 
Equation (\ref{eq:Tnrec}) and
that $b_n$ counts ordered pairs of simple $3$-threshold sequences. The first result allows us to deduce that
simple $3$-threshold sequences are in bijection with ternary trees, and the second result ensures the bijection between
double $3$-threshold sequences and ordered pairs of ternary trees. 

For each $i$ with $1\leq i\leq n-1$, let the values $3i,3i+1,3i+2$ be called {\em standard values} for $s_i$, and the values
larger than $3i+2$ be called {\em high values} for $s_i$. For $s_n$, there are only standard values, either one or two depending
whether the threshold sequence is simple or double. See Figure \ref{fig:standard}, with $k=3$.

\subsection{First count: from left to right} 
A simple $3$-threshold sequence of length $n$ can be of two types with respect to $s_1$. 

\begin{itemize}
 \item  When $s_1$ is a standard value, 
$(s_2 \ldots s_n)$ is any simple $3$-threshold sequence of length $n-1$ and offset $3$. Using Remark \ref{rem:decalage},  
the number of such sequences is $a_{n-1}$, thus the number of simple $3$-threshold sequence of length $n$ with these properties
is $3a_{n-1}$, since there are three standard values available for $s_1$.
\item When $s_1$ is a high value, there exists a largest value $h$ such that $s_1, s_2, \ldots, s_h$ are high values. Then $s_{h+1}$ 
is a standard value, and $1\leq h\leq n-2$, since $s_n=kn$ and thus no high value is available for $s_{n-1}$. The simple $3$-threshold sequence
is thus made of a subsequence $(s_1 \ldots s_{h+1})$ of length $h+1$ containing only high values except for the last one, 
which is a standard value, followed by a subsequence $(s_{h+2} \ldots s_n)$ of length $n-h-1$ which may be any simple $3$-threshold sequence 
of this length with offset $3(h+1)$. Using again Remark \ref{rem:decalage}, the number of 
such $3$-threshold sequences for a fixed $h$ is thus  $c_{h+1}a_{n-h-1}$, where $c_{h+1}$ is the number of sequences $(s_1 \ldots s_{h+1})$ described
above. Each such sequence satisfies $3(i+1)\leq s_i\leq 3(h+1)+2$ for $1\leq i\leq h$
and $3(h+1)\leq s_{h+1}\leq 3(h+1)+2$. Thus its $h-1$ first elements are constrained as in a $3$-threshold sequence with offset $3$, 
whereas $s_{h}, s_{h+1}\in\{3(h+1), 3(h+1)+1, 3(h+1)+2\}$. Then we cannot have $s_{h+1}=3(h+1)$, since no value would be available for $s_{h}$. Moreover, when
$s_{h+1}=3(h+1)+1$ we have $s_h=3(h+1)$ and $(s_1 \ldots s_{h-1} s_h)$ is any simple $3$-threshold
sequence with offset $3$ of length $h$; there are $a_{h}$ such sequences. And when $s_{h+1}=3(h+1)+2$ we have $s_h\in\{3(h+1), 3(h+1)+1\}$ 
and $(s_1 \ldots s_{h-1} s_h)$ is any double $3$-threshold sequence with offset $3$ of length $h$; there are $b_{h}$ such sequences.
We deduce that $c_{h+1}=a_h+b_h$. 
\end{itemize}

In conclusion, the number of simple $3$-threshold sequences of length $n$ is given by:

\begin{equation}
 a_n=3a_{n-1}+\sum \limits_{h=1}^{n-2}(a_h+b_h)a_{n-h-1}
 \label{eq:anmixte}
\end{equation}

For the double $3$-threshold sequences, the enumeration is similar, except that we may have $h=n-1$ when $s_1$ is a high value.
This can only happen for sequences with $s_{n-1}=3n$ and $s_n=3n+1$. For these sequences, $(s_1 \ldots s_{n-1})$ is any
simple $3$-threshold sequence, so an amount of $a_{n-1}$ must be added to the final count:

\begin{equation}
 b_n=3b_{n-1}+\sum \limits_{h=1}^{n-2}(a_h+b_h)b_{n-h-1}+a_{n-1}
 \label{eq:bnmixte}
\end{equation}

\subsection{Second count: from right to left}

Consider a simple $3$-threshold sequence. Recall that $s_n=3n$ and make the following observation.

\begin{remark}
For every $j\leq n-1$, the sequences $(s_1 s_2 \ldots s_j)$ satisfying for each $i$ with $1\leq i\leq j$ the constraint 
$3i\leq s_i\leq 3j+2$ ({\em i.e.} such that $s_j$ is a 
standard value) are exactly the $j$-length prefixes of the simple $3$-threshold sequences of length $j+1$. 
Thus, there exist $a_{j+1}$ such sequences.
\label{rem:ajplus1}
\end{remark}

Again, a simple $3$-threshold sequence may be of two types.
\begin{itemize}
 \item When $s_{n-1}\in\{3(n-1), 3(n-1)+1\}$, the sequence $(s_1 \ldots s_{n-1})$ is any double $3$-threshold sequence of length $n-1$. 
 There are $b_{n-1}$ such simple $3$-threshold sequences of length $n-1$. 
 \item When $s_{n-1}=3(n-1)+2$,  two cases are possible. When $s_{n-2}$ is a standard value, by Remark \ref{rem:ajplus1} we deduce
 that there are  $a_{n-1}$ simple $3$-threshold sequences of length $n$ with this property. When $s_{n-2}$ is a high value, we have
 $s_{n-2}\in\{3(n-1), 3(n-1)+1\}$ and there exists 
 a lowest index $h$, with $0\leq h\leq n-3$, such that $s_{h+1}, \ldots, s_{n-2}$ are high values. Thus $s_{h}$, when $h\geq 1$, 
 is a standard value. In this case, the simple $3$-threshold sequence of length $n$ is made of a subsequence $(s_1 s_2 \ldots s_h)$
 satisfying Remark \ref{rem:ajplus1} for $j=h$, followed by a double $3$-threshold sequence $(s_{h+1} s_{h+2} \ldots s_{n-2})$ with offset 3 and 
 length $n-2-h$, which is on its turn followed by $3(n-1)+2$ and $3n$. The number of such sequences, for a fixed $h$,
 is $a_{h+1}b_{n-2-h}$.
\end{itemize}

We deduce that:

\begin{equation}
 a_{n}=b_{n-1}+a_{n-1}+\sum\limits_{h=0}^{n-3}a_{h+1}b_{n-2-h}=b_{n-1}+a_{n-1}+\sum\limits_{h=1}^{n-2}a_{h}b_{n-1-h}=\sum\limits_{h=0}^{n-1}a_{h}b_{n-1-h}
\label{eq:anfctbn}
 \end{equation}

A double $3$-threshold sequence is in one of the two following cases:

\begin{itemize}
 \item When $s_{n}\in\{3n, 3n+1\}$ and $s_{n-1}$ is a standard value, the double $3$-threshold sequence of length $n$ is made of
 a subsequence $(s_1 s_2 \ldots s_{n-1})$ satisfying Remark \ref{rem:ajplus1}, followed either by $3n$ or by $3n+1$. There
 are $2a_n$ such sequences.
 \item When $s_n=3n+1$ and $s_{n-1}=3n$, two situations are possible for the sequence $(s_1 s_2 \ldots s_{n-2})$. If $s_{n-2}$
 is a standard value, by Remark \ref{rem:ajplus1} there are $a_{n-1}$ simple $3$-threshold sequences $(s_1 s_2 \ldots s_{n-2})$,
 and thus the same number of double $3$-threshold sequences with the desired properties. If $s_{n-2}$ is a high value, then
 we consider as before the minimum value of $h$, $0\leq h\leq n-3$, such that $s_{h+1}, \ldots, s_{n-2}$ are high values.
 The  double $3$-threshold sequence of length $n$ is made of
 a subsequence $(s_1 s_2 \ldots s_{h})$ satisfying Remark \ref{rem:ajplus1}, followed by a simple $3$-threshold sequence 
 $(s_{h+1} s_{h+2}$ $\ldots$ $s_{n-1})$ with offset 3, and further
 followed by $3n+1$. For a fixed $h$, the number of such sequences is $a_{h+1}a_{n-1-h}$.
 \end{itemize}
 
 In conclusion, we have:
 
 \begin{equation}
 b_{n}=2a_{n}+a_{n-1}+\sum\limits_{h=0}^{n-3}a_{h+1}a_{n-1-h}=2a_{n}+\sum\limits_{h=1}^{n-1}a_{h}a_{n-h}=\sum\limits_{h=0}^{n}a_{h}a_{n-h}
 \label{eq:bnfctan}
\end{equation}
 
\subsection{Implicit bijections}
 
\begin{proposition}
Simple 3-threshold sequences of length $n$ are in bijection with ternary trees with $n$ internal nodes. 
\label{prop:simple}
\end{proposition}

\begin{proof}
 We use Equations (\ref{eq:anmixte}) and (\ref{eq:bnfctan}) to deduce that $a_n$ satisfies the same recurrence
 relation as $T_n$ (see Equation (\ref{eq:Tnrec}), where $C_n^{(3)}=T_n$).
 
\begin{align*}
  a_n& =3a_{n-1}+\sum \limits_{h=1}^{n-2}(a_h+b_h)a_{n-h-1}\\
    &=3a_{n-1}+\sum \limits_{h=1}^{n-2}a_ha_{n-h-1}+\sum \limits_{h=1}^{n-2}b_ha_{n-h-1}\\
    &=3a_{n-1}+\sum \limits_{h=1}^{n-2}a_ha_{n-h-1}+\sum \limits_{h=1}^{n-2}(\sum\limits_{j=0}^{h}a_ja_{h-j})a_{n-h-1}\\
    &=a_{n-1}+
    \sum \limits_{j=0}^{n-1}a_ja_{n-j-1}a_{n-(n-1)-1}+\sum \limits_{h=1}^{n-2}\sum\limits_{j=0}^{h}a_ja_{h-j}a_{n-h-1}\\
    &=a_{n-1}+\sum \limits_{h=1}^{n-1}\sum\limits_{j=0}^{h}a_ja_{h-j}a_{n-h-1}
    =\sum \limits_{h=0}^{n-1}\sum\limits_{j=0}^{h}a_ja_{h-j}a_{n-h-1}\\
    &=\sum \limits_{j=0}^{n-1}\sum\limits_{h=j}^{n-1}a_ja_{h-j}a_{n-h-1}
    = \sum \limits_{j=0}^{n-1}\sum\limits_{i=0}^{n-1-j}a_ja_{i}a_{n-j-i-1}
    =\sum_{\substack{j+i+h=n-1\\ 0\leq j,i,h\leq n-1}}^{}a_ja_ia_h
\end{align*}

Since $a_0=T_0=1$ and $a_n,T_n$ satisfy the same recurrence relation, we deduce that $a_n=T_n$ 
for $n\geq 1$. The conclusion follows.
\end{proof}

 \begin{proposition}
Double 3-threshold sequences of length $n$ are in bijection with ordered pairs of ternary trees whose total number of internal nodes
is $n$. 
\label{prop:double}
\end{proposition}

\begin{proof}
By Equation (\ref{eq:bnfctan}), we have that  $b_{n}=\sum\limits_{h=0}^{n}a_{h}a_{n-h}$. For each $h$, the term $a_ha_{n-h}$
counts the ordered pairs of simple $3$-threshold sequences where the first (second) sequence in the pair is any
simple $3$-threshold sequence of length $h$ ($n-h$). But, by Proposition \ref{prop:simple},
simple $3$-threshold sequences of length $w$ are in a 1-to-1 correspondence with ternary trees with $w$ internal nodes, so
the conclusion follows.
\end{proof}

The two previous propositions imply that:

\begin{corollary}
For each $n\geq 1$, $a_n=T_n=\frac{1}{2n+1}\binom{3n}{n}$ and $b_n=U_n=\frac{1}{n+1}\binom{3n+1}{n}$.
\label{cor:anTn}
\end{corollary}

\begin{proposition}
 Proper double $3$-threshold sequences are in bijection with ordered 4-tuples of ternary trees with total number of internal nodes equal to 
 $n-1$. Moreover, we have

 \begin{equation}
b_{n}-a_n=\sum\limits_{h=0}^{n-1}a_{h}a_{n-h}=\sum\limits_{h=0}^{n-1}b_{h}b_{n-h-1}=U_n-T_n=\frac{2}{n+1}\binom{3n}{n-1}.
\end{equation}

 \label{prop:properdouble}
\end{proposition}

\begin{proof}
%
%
%

Using Equations (\ref{eq:bnmixte}) and (\ref{eq:anfctbn}), we obtain:

\begin{align*}
b_n&=3b_{n-1}+\sum\limits_{h=1}^{n-2}a_{h}b_{n-1-h}+\sum\limits_{h=1}^{n-2}b_{h}b_{n-1-h}+a_{n-1}\\
&=3b_{n-1}+(a_n-b_{n-1}-a_{n-1})+\sum\limits_{h=1}^{n-2}b_{h}b_{n-1-h}+a_{n-1}\\
&=a_n+2b_{n-1}+\sum\limits_{h=1}^{n-2}b_{h}b_{n-1-h}\\
&=a_n+\sum\limits_{h=0}^{n-1}b_{h}b_{n-1-h}
\end{align*}

\noindent which implies together with Equation (\ref{eq:bnfctan}) that:

\begin{equation}
b_{n}-a_n=\sum\limits_{h=0}^{n-1}b_{h}b_{n-h-1}=\sum\limits_{h=0}^{n-1}a_{h}a_{n-h}
\end{equation}

Now, $\sum\limits_{h=0}^{n-1}b_{h}b_{n-h-1}=\sum\limits_{h=0}^{n-1}U_{h}U_{n-h-1}$ and thus counts, for $n\geq 1$, 
the number of ordered 4-tuples of ternary trees with total number of internal nodes equal to $n-1$. 
This follows from the interpretation of $U_k$, $k\geq 0$, which counts the pairs of ternary trees with $k$ internal nodes. 
A simple computation further shows, using Corollary \ref{cor:anTn}, that $b_n-a_n=U_n-T_n=\frac{2}{n+1}\binom{3n}{n-1}$.
The conclusion follows. \end{proof}

\begin{remark}
The sequence corresponding to the number of ordered 4-tuples of ternary
trees is known as OEIS A006629 \cite{oeis}.
\end{remark}

\section{Explicit bijections}\label{sec:explainbijections}

We start this section by defining a labeling for the $k$-ary trees. Note that a $k$-ary tree with $n$ internal nodes has
$nk+1$ nodes.

Let $w$ be an integer. A {\em $k$-ary $w$-tree} is a $k$-ary tree $A$ whose nodes are labeled such that a breadth first 
traversal of $A$ yields the list of nodes $w, w-1, \ldots, w-nk$, where $n$ is the number of internal nodes of $A$. 
Equivalently, the root of $A$ is $w$, and the nodes on each level are labeled in decreasing order
from left to right, starting with the largest value not used on the previous level.

\begin{example}
 The trees in Figure \ref{fig:4arbres} are respectively a quaternary $16$-tree, a quaternary $18$-tree and a quaternary $9$-tree.
\end{example}

\begin{remark}
 Two $k$-ary $w$-trees are equal if and only if the sets of their internal nodes, seen as labeled nodes, are equal.
 \label{rem:equalwtrees}
\end{remark}

\begin{remark}
In each $k$-ary $w$-tree with $j$ internal nodes, the smallest label $x$ of a node satisfies 
$x=w-jk$.
\label{rem:smallestx}
\end{remark}

Let  $S=(s_1 s_2\ldots s_n)$ be a $(k,l)$-threshold sequence.  The {\em cut index} of $S$ is the largest 
index $i<n$ such that 

\begin{equation}
s_i<s_n-(n-i)k
\label{eq:cutting}
\end{equation}

\noindent if such an element exists,
and is equal to $0$ otherwise.  Intuitively, $s_i$ is the largest element in $S$ whose value is not large enough to 
be a label in the $k$-ary $s_n$-tree whose internal nodes have the labels $s_n, s_{n-1}, \ldots, s_{i+1}$ (see
also Remark \ref{rem:smallestx}).
For each $(k,l)$-threshold sequence $S$ of length $n$,
let  $Forest(S)$ be the set of trees, initially empty, defined as follows (see Figure \ref{fig:4arbres}, as well as Examples \ref{ex:exempleForest0} and \ref{ex:exempleForest}):

\begin{itemize}
 \item[(1)] Let $A_S^1$ be the  $k$-ary $s_n$-tree whose internal nodes ordered according to a breadth first traversal are the elements 
 $s_n, s_{n-1}, \ldots, s_{i+1}$ of $S$. Then $A_S^1$ belongs to $Forest(S)$ (discard the node labels, 
 they were needed only to define the tree).
 \item[(2)] If $i\neq 0$, let $Q=S[1,i]$ and add  $Forest(Q)$ to $Forest(S)$.
\end{itemize}

Assume $Forest(S)=\{A_S^1, A_S^2,\ldots, A_S^t\}$ where $1, 2, \ldots, t$ indicate the order of computation of the trees.
Let $Q_1=S$ and let $Q_2, \ldots, Q_{t}$ be the successive sequences computed in step (2) above, respectively 
generating $A_S^1, A_S^2,$ $\ldots, A_S^t$ using step (1).  We prove below that the unique index $l_p$ such that $Q_p$ is a proper 
$(k,l_p)$-threshold sequence, $1\leq p\leq t$, satisfies $l=l_1>l_2> \ldots > l_t$. Then define $Tuple(S)$ as the $(l+1)$-tuple containing 
$A_S^p$ in position $l_p+1$, $1\leq p\leq t$, and the trivial $k$-ary tree $\lambda$ on each of the remaining positions.

\begin{figure*}[t!]
    \begin{subfigure}[t]{0.36\textwidth}
        \centering
        \includegraphics[width=0.92\textwidth]{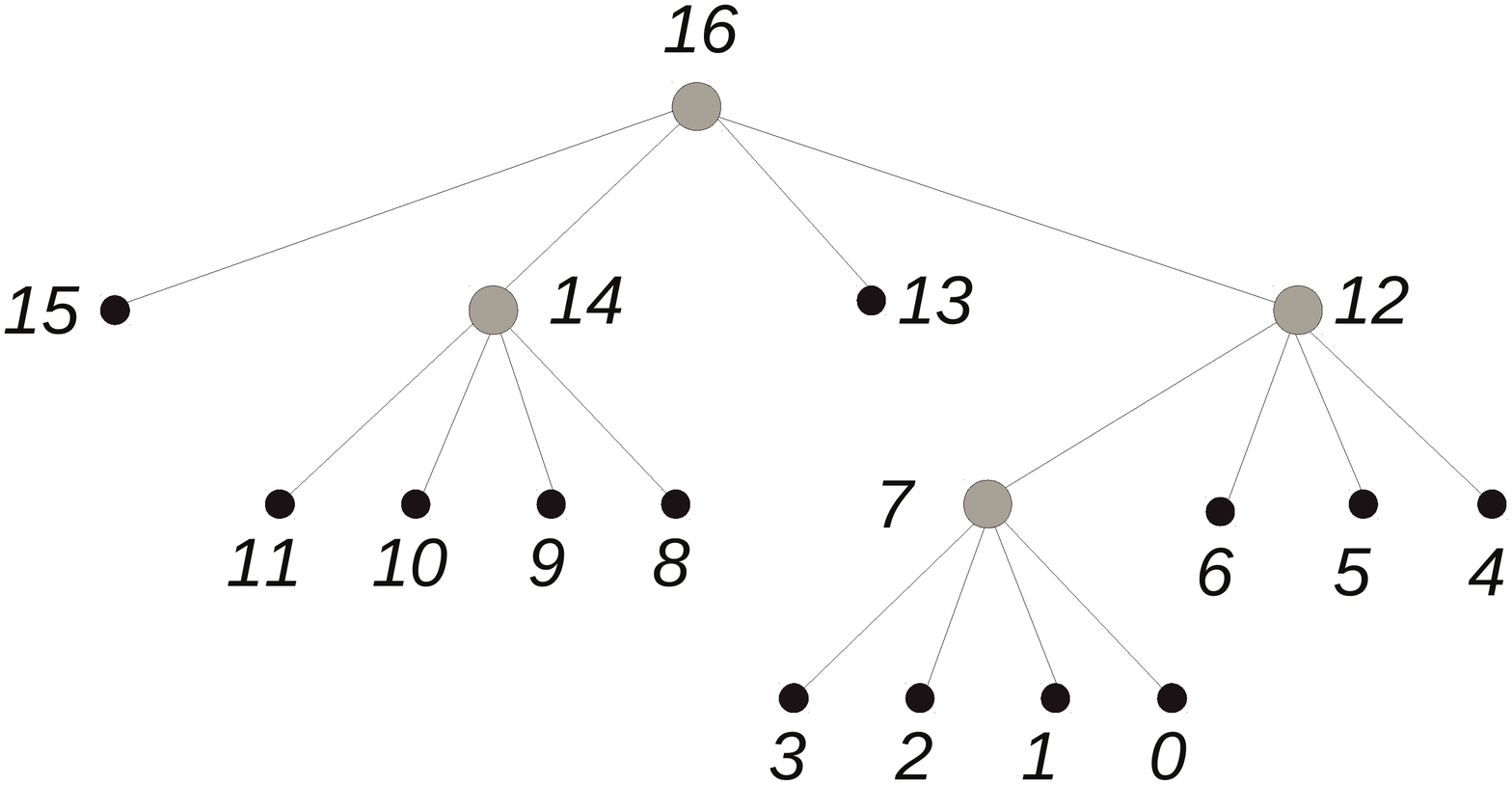}
        \caption{}
    \end{subfigure}\hfill
    \begin{subfigure}[t]{0.57\textwidth}
        \centering
        \vspace*{-5cm}
        \includegraphics[width=0.95\textwidth]{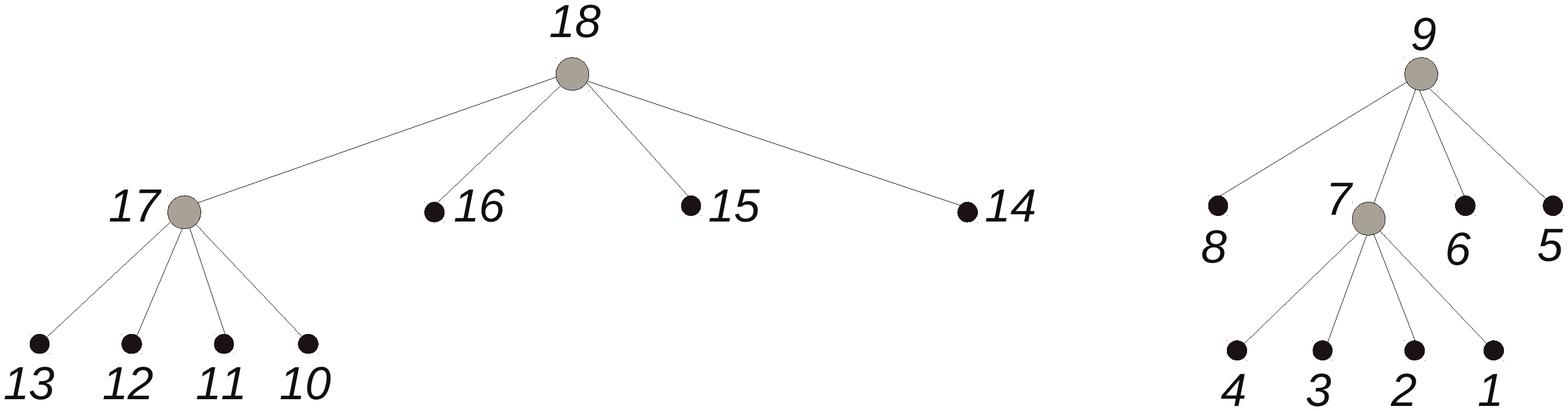}
        \vspace*{-1.5cm}
        \caption{}   
    \end{subfigure}
    \caption{ a) The quaternary tree $A_S^1$ computed for the $(4,0)$-threshold sequence $S=(7,12,14,16)$.
    b) The two quaternary trees $A_V^1$ (left), $A_V^2$ (right) computed for the $(4,2)$-threshold sequence $V=(7,9,17,18)$.}
    \label{fig:4arbres}
    \end{figure*}

\begin{example}
 For the proper $(4,0)$-threshold sequence $S=(7,12,14,16)$ of length $4$, also denoted by $Q_1$, none of the indices satisfies the definition of a cut index, thus
 $Forest(S)=\{A_S^1\}$ and the quaternary tree $A_S^1$ is depicted in Figure 
 \ref{fig:4arbres}(a). The $(l+1)$-tuple associated with $S$ is a 1-tuple, represented by $A_S^1$ alone. When $S$ is seen as a $(4,2)$-threshold
 sequence, the $(l+1)$-tuple associated with $S$ is the 3-tuple $(A_S^1,\lambda,\lambda)$. Indeed,  $16=4\cdot 4+0$ thus 
 for the sequence $V(=Q_1)$ we have $l_1=0$.
  \label{ex:exempleForest0}
 \end{example}
 
 \begin{example}
 For the proper $(4,2)$-threshold sequence $V=(7,9,17,18)$, also denoted by $Q_1$, we have $s_3=17$. Then $s_3$ does not satisfy the condition 
 (\ref{eq:cutting}), since  $17\not <18-(4-3)4=14$, so the index $3$ is not 
 the cut index of $V$.
 But $s_2=9$ and $9<18-(4-2)4=10$, so that $9$ does not belong to the tree with internal nodes 18 and 17, denoted by $A_V^1$ and
 depicted in Figure \ref{fig:4arbres}(b), left. Therefore $i=2$ is
 the cut index of $V$, and $9$ is the root of the next tree. The next tree, called $A_V^2$ is thus computed in the same way using
 the $(4,2)$-threshold sequence $Q_2=V[1,i]=(7,9)$, which has no cut index. Then $A_V^2$ is the tree in Figure \ref{fig:4arbres}(b), right.
 We deduce that $Forest(V)=\{A_V^1,A_V^2\}$.
 The $(l+1)$-tuple associated with $V$ is the 3-tuple $(\lambda,A_V^2,A_V^1)$  
 since $18=4\cdot 4+2$ (thus $l_1=2$ for $Q_1=V$) and $9=2\cdot 4 +1$ (thus $l_2=1$ for $Q_2=V[1,2]=(7,9)$).
 \label{ex:exempleForest}
\end{example}

\begin{theorem}
 Let $n,k,l$ be three integers with $n\geq 1$, $k\geq 2$ and $0\leq l\leq k-2$. The function $Tuple(S)$ is well-defined for the
 pair $(k,l)$, and is a 
 bijection between $(k,l)$-threshold sequences of length $n$ and $(l+1)$-tuples of $k$-ary trees with total number of $n$ internal vertices.
 \label{thm:main}
\end{theorem}

\begin{proof}
We first prove two affirmations, named (A) and (B):
\medskip

(A) For any $(k,l)$-threshold sequence $S$ of length $n$, the number of trees in $Forest(S)$ is upper bounded by $l+1$.
\medskip

When $A_S^1$ is built, the value $s_i$ given by the cut index $i$ of $S$ is not a node of $A_S^1$. Indeed,
$s_i$ is the $(n-i+1)$-th element of $S$ in decreasing order of the indeces, thus $A_S^1$ is built on $(n-i)$ internal nodes and
by Remark \ref{rem:smallestx} its smallest label is $s_n-(n-i)k$. Condition (\ref{eq:cutting}) then ensures that
$s_i$ is not a node label from $A_S^1$.  Then, either $A_S^1$ contains all the elements in $S$ (this is the case $i=0$), or
$S[1,i]$ must be used to complete the set $Forest(S)$ (this is the case $i\geq 1$). 

We use induction on $l$ to show Affirmation (A). When $l=0$, $s_n=kn$ and for each $r<n$ we have $s_r\geq kr=kn-(n-r)k=s_n-(n-r)k$, 
thus $s_r$ belongs to the $k$-ary $s_n$-tree with internal nodes $s_n, s_{n-1}, \ldots, s_{r+1}$.
We deduce that $A_S^1$ contains all the elements in $S$ and the conclusion follows.

Assume now that Affirmation (A) is true for all $(k,l')$-threshold sequences $Q$ with $l'<l$. Assume moreover that $S$ is a 
proper $(k,l)$-threshold sequence, {\em i.e.} $s_n=kn+l$, otherwise the
conclusion follows by inductive hypothesis.
Let $i$ be the cut index of $S$, denote by $Q=S[1,i]$ and note that $q_i=s_i< s_n-(n-i)k=kn+l-nk+ik=ik+l$, thus  
$q_i\leq ik+l-1$. This means $Q$ is a proper $(k,l_Q)$-threshold sequence of length $i$, for some  $l_Q\leq l-1$. Then, using the inductive hypothesis
for $Q$, we deduce that the number of trees in $Forest(Q)$ is upper bounded by $l_Q+1$, thus by $l$, and therefore the number of trees in 
$Forest(S)$ is upper bounded by $l+1$.
\medskip 

(B) With the notation $Q_1=S, Q_2, \ldots, Q_{t}$ for the successive sequences respectively generating $A_S^1,$ $A_S^2,$ $\ldots, A_S^t$,
and assuming each $Q_p$ is a proper $(k,l_p)$ sequence, the sequence 
$l_1, l_2, \ldots, l_t$ is a strictly decreasing sequence of integers.
\medskip

As proved above for $S$ and $Q=S[1,i]$, where $i$ is the cut index of $S$, we have $q_i=s_i< ik+l$ meaning that
$q_i=ik+l'$ with $l'<l$. The same reasoning may be applied to each pair $Q_i, Q_{i+1}$ and Affirmation (B) follows.
\medskip

By Affirmations (A) and (B), the function $Tuple(S)$ is well defined, since no pair of trees $A_S^p, A_S^r$ is affected to the same position, and 
there are enough positions in the $(l+1)$-tuple to contain all the trees in $Forest(S)$.
\medskip

We show that $Tuple(S)$ is a bijection.
\medskip

{\bf Injectivity.} Assume that $Tuple(S)=Tuple(V)$ for two $(k,l)$-threshold sequences $S$ and $V$. Denote by $y_1, \ldots, y_t$ from {\em right}
to
{\em left} the  positions of the non-trivial trees, which are the same in both tuples. Then by the definition of the function $Tuple$
and the equation $Tuple(S)=Tuple(V)$, we have that $A_S^p=Tuple(S)[y_p]=Tuple(V)[y_p]=A_V^p$ for $1\leq p\leq t$. Recall that all the trees in the tuples
are unlabeled $k$-ary trees.

%
Let $r_{p-1}$ be the number of internal nodes in
$A_S^1, \ldots, A_S^{p-1}$. By the definition of the function $Tuple$, and using the same notations $Q_i$ as above with respect to $S$,
during the construction of $Tuple(S)$ the root of $A_S^p$ has the label $w_p=k(n-r_{p-1})+l_p$. The same reasoning for $V$,
assuming the sequences used to build $Tuple(V)$ are $Y_1, \ldots, Y_t$, imply that the root $A_V^p$ has the label 
$k(n-r_{p-1})+ l'_p$, where $Y_p$ is a proper $(k,l'_p)$-sequence.  But $A_S^p$ and $A_V^p$ are placed on the same position 
in $Tuple(S)$ and $Tuple(V)$ respectively, meaning that $l_p=l'_p=y_p-1$, and thus the roots of $A_S^p$ and $A_V^p$ have
the same label $w_p$. Then $A_S^p$ and $A_V^p$ are not only identical
when seen as $k$-ary trees, but also when seen as $k$-ary $w$-trees (see Remark \ref{rem:equalwtrees}).

Then let $I^p$ be the  increasing sequence of the internal nodes of $A_S^p$, seen as a $k$-ary $w_p$-tree. Then $S=I_tI_{t-1}\ldots I_1$.
The sequence $V$ is computed similarly using the same increasing sequences, since the labeled trees are identical. Thus $S=V$.
\medskip

{\bf Surjectivity.} Let $A=(A_1, \ldots, A_{l+1})$ be a $(l+1)$-tuple of $k$-ary trees, and assume that $A_{y_1}, A_{y_2}, \ldots, A_{y_t}$
are the non-trivial $k$-ary trees, with $y_1>y_2>\ldots>y_t$. Let
$r_{p}$ be the total number of internal nodes in $A_{y_1}, \ldots, A_{y_p}$. We label $A_{y_p}$ as a $k$-ary $w_p$-tree
with $w_p=k(n-r_{p-1})+ y_p-1$, and we denote by $I_p$ the increasing sequence of its internal nodes. Then 
the length of $I_p$ is $|I_p|=r_{p}-r_{p-1}$.

Define $S=I_t\ldots I_1$ and notice that the number $n$ of elements in $S$ is the total number $n$ of internal vertices 
in all the trees of $A$. Also denote by $Q_p=I_t\ldots I_{p}$ for $1\leq p\leq t$. Then the rightmost element of $Q_p$ is the root of 
$A_{y_p}$, that is $w_p$,  and its length is:

$$|Q_p|=n-(|I_1|+\ldots +|I_{p-1}|)=n-(r_1-r_0+r_2-r_1+\ldots + r_{p-1}-r_{p-2})=n-r_{p-1}.$$ 

Then we also have $w_p=s_{|Q_p|}.$
We note that: 
\medskip
\begin{align*}
w_p&=k(n-r_{p-1})+ y_p-1=k(n-r_{p-2})-(r_{p-1}-r_{p-2})k+y_{p-1}-1+y_p-y_{p-1}\\
&=w_{p-1}-(r_{p-1}-r_{p-2})k+y_p-y_{p-1}\\
&< w_{p-1}-(r_{p-1}-r_{p-2})k=w_{p-1}-(|Q_{p-1}|-|Q_p|)k.
\end{align*}

Now, $w_p<w_{p-1}-(|Q_{p-1}|-|Q_p|)k$ is condition (\ref{eq:cutting}) applied to $Q_{p-1}$ and the index $|Q_p|$ of $Q_{p-1}$, 
given that $w_p=s_{|Q_p|}$. 
Moreover, by the definition of a $k$-ary $w_{p-1}$-tree, for any element $s_j$ situated between $w_{p-1}$ and $w_p$
Equation (\ref{eq:cutting}) cannot hold. Thus $|Q_p|$ is the cut index of $Q_{p-1}$, and thus $A=Tuple(S)$.
\end{proof}

\begin{corollary}
 The number of $(k,l)$-threshold sequences of length $n$ is equal to the Raney number:
 
 $$R_n^{(k,l+1)}=\frac{l+1}{kn+l+1}\binom{kn+l+1}{n}. $$
 \label{cor:numberthresholdseq}
\end{corollary}

\begin{proof}
 The result follows from Theorem \ref{thm:main},  Remark \ref{ex:exempleForest} and Equation (\ref{eq:Raney}).
\end{proof}

\begin{corollary}
For each $l\geq 1$, the number of proper $(k,l)$-threshold sequences of length $n$ is equal to the Raney number:

$$R_{n-1}^{(k,k+l)}=\frac{k+l}{(k-1)(n-1)+k+l}\binom{kn+l-1}{n-1}.$$
\label{cor:properseq}
\end{corollary}

\begin{proof}
By Corollary \ref{cor:numberthresholdseq}, the number of proper $(k,l)$-threshold sequences of length $n$ is equal to $R_n^{(k,l+1)}-R_n^{(k,l)}$, and a well known recurrence 
 relation or a simple verification indicates that this value is  $R_{n-1}^{(k,k+l)}$.
\end{proof}

When $l=0$, all the $(k,0)$-threshold sequences are proper, so the number of proper $(k,0)$-threshold sequences is $R_n^{(k,1)}$,
according to Corollary \ref{cor:numberthresholdseq}.

\section{Combinatorial identities}\label{sec:combinatorailidentities}

In this section, we deduce from our previous results three combinatorial identities, obtained when $k=2$ and $k=3$.

\subsection{Case $k=2$: Catalan numbers}

\begin{proposition}
 Catalan numbers $C_n$ satisfy for all $n\geq 1$ the recurrence relation:
 
 \begin{equation}
  C_n= \sum_{\substack{r+s+t=n-1\\ r,s\geq 1;\, t\geq 0}}^{}C_rC_s2^t+2^{n-1}
 \label{eq:recCatalan}
 \end{equation}

\end{proposition}

\begin{proof}
 By Corollary \ref{cor:numberthresholdseq} for $k=2$, $l=0$ and recalling that $C_n=C_n^{(2)}=R_n^{(2,1)}$, we deduce that 
 $C_n$ is the number of $(2,0)$-threshold sequences of length $n$.
 
 As in Section \ref{sec:three}, we propose a count of the $(2,0)$-threshold sequences $S$ of length $n$ that
 allows us to obtain Equation (\ref{eq:recCatalan}). The last element of
 these sequences is always $s_n=2n$. Similarly to the case $k=3$, for each $i$ with $1\leq i\leq n-1$, the values 
 $2i,2i+1$ are called {\em standard values} for $s_i$, and the values larger than $2i+1$ are called {\em high values} for $s_i$. 
 See Figure \ref{fig:standard}. We assume that $s_0=0$, and this value is neither high nor standard.
 
 For each $(2,0)$-threshold sequence $S$ of length $n$, we introduce the following notation:
 \begin{align*}
a&=\max\{j\, |\,s_{j-1} \text{ is not a standard value and } s_{j}   \text{ is a standard value}\}\\
b&=\max\{j\, |\,s_j \text{ is not a high value value and } s_{j+1} \text{ is a high value}\}
 \end{align*}
 
  Then $a=1$ if and only if the sequence $S$ contains only standard values. There are $2^{n-1}$ such $(2,0)$-threshold sequences.
  
  Otherwise, $2\leq a\leq n-1$ since there are no high values for $s_{n-1}$. In this case, $S$ is made of 
  1) a $(2,0)$-threshold subsequence of length $a$ such that $s_{a-1}=2a$ and $s_a=2a+1$ (these are the only possible values),
 concatenated with 2) any subsequence $(s_{a+1} \ldots s_{n-1})$ of length $n-1-a$ made of standard elements, and followed by $2n$. 

 The subsequences described in 2) are counted by $2^{n-1-a}$ since each $s_i$ may take one of two precise
 values, for each $i$ with $a+1\leq i\leq n-1$. We need $b$ to count the subsequences described in 1). 
 We have $0\leq b\leq a-2$. Each of the subsequences described in 1) are made of 1') an arbitrary
 $(2,0)$-subsequence $(s_1 \ldots s_b)$ with $s_b\in \{2b,2b+1\}$, followed by 2') an arbitrary $(2,0)$-subsequence 
 $s_{b+1}, \ldots, s_{a-1}$ with $s_i\geq 2i+2$ for each $i$ with $b+1\leq i\leq a$ and $s_{a-1}=2a$ as explained above; the sequence 
 ends with $s_a=2a+1$.  By Remark \ref{rem:ajplus1},  the subsequences in 1') are in bijection with the $(2,0)$-threshold sequences of length $b+1$.  Then these
 sequences are enumerated by $C_{b+1}$. The subsequences in 2') are the $(2,0)$-threshold sequences of length $a-1-b$ and offset 2.
 
Thus the total number of $(2,0)$-threshold sequences of length $n$ is
 \begin{equation}
  C_n= \sum \limits_{a=2}^{n-1}\sum \limits_{b=0}^{a-2}C_{b+1}C_{a-1-b}2^{n-1-a}+2^{n-1}
  \label{eq:13}
 \end{equation}

 Let $r=b+1, s=a-r$ and $t=n-1-(r+s)$. Then we successively have:
 
 \begin{align*}
   C_n&= \sum \limits_{a=2}^{n-1}\sum \limits_{r=1}^{a-1}C_{r}C_{a-r}2^{n-1-a}+2^{n-1}
   =\sum \limits_{r=1}^{n-2}\sum \limits_{a=r+1}^{n-1}C_{r}C_{a-r}2^{n-1-a}+2^{n-1}\\
   &=\sum \limits_{r=1}^{n-2}\sum \limits_{s=1}^{n-1-r}C_{r}C_{s}2^{n-1-(s+r)}+2^{n-1}
   = \sum_{\substack{r+s+t=n-1\\  r,s\geq 1;\, t\geq 0}}^{}C_rC_s2^t+2^{n-1}
 \end{align*}

The conclusion is proved. 
\end{proof}


\subsection{Case $k=3$: Relations between $T_n$ and $U_n$}

According to the introduction and to Corollary   \ref{cor:anTn}, we have: 

\begin{enumerate}[itemsep=0pt]
\item $T_n=a_n=C_n^{(3)}$, and this value counts the number of ternary trees with $n$ internal nodes, as well as the simple 3-threshold sequences of length $n$.

\item $U_n=b_n=R_n^{(3,2)}$, and this value counts the number of ordered pairs of ternary trees with a total of $n$ internal nodes, 
as well as the double 3-threshold sequences of length $n$.
\end{enumerate}

\begin{proposition}
 Assuming $T_0=U_0=1$, the following relations hold for $n\geq 1$:
 
\begin{align*}
 2\sum \limits_{h=0}^{n-1}\frac{1}{h+1}T_hT_{n-h-1}& =3U_{n-1}-T_n\\
 2\sum \limits_{h=0}^{n-1}\frac{1}{3h+1}U_hU_{n-h-1}& =4T_{n}-U_n
\end{align*}
 \label{prop:lastrecs}
\end{proposition}

\begin{proof}
 We use the equations proved in Section \ref{sec:three}, and thus also the notations $a_n$ and $b_n$, rather than $T_n$ and $U_n$,  
 during the proof.
 
 For each $n\geq 1$, the following equations may be obtained by basic computations using the close forms for $a_n$ and $b_n$: 
 
 \begin{align*}
 a_n+b_n&=\frac{2}{n+1}\binom{3n}{n}=\frac{2}{n+1}(2n+1)a_n=2(2-\frac{1}{n+1})a_n\\
  \frac{1}{n+1}a_n&=\frac{1}{3n+1}b_n\\
 3a_n-b_n&=\frac{1}{n}(b_n-a_n)=\frac{2}{3n+1}b_n
\end{align*}

 They are used below without recalling it.
 Then Equation (\ref{eq:anmixte}) and Equation (\ref{eq:bnfctan}) imply:

 \begin{align*}
 a_n-3a_{n-1}&= \sum\limits_{h=1}^{n-2}(a_h+b_h)a_{n-h-1}=2\sum\limits_{h=1}^{n-2}(2-\frac{1}{h+1})a_ha_{n-h-1}\\
 &=4\sum\limits_{h=1}^{n-2}a_ha_{n-h-1}-2\sum\limits_{h=1}^{n-2}\frac{1}{h+1}a_ha_{n-h-1}\\
 &=4(b_{n-1}-2a_{n-1})-2\sum\limits_{h=1}^{n-2}\frac{1}{h+1}a_ha_{n-h-1} 
 \end{align*}
 
 We deduce:
 \begin{align*}
  2\sum\limits_{h=1}^{n-2}\frac{1}{h+1}a_ha_{n-h-1}&=4b_{n-1}-a_n-5a_{n-1}\\
  2\sum\limits_{h=0}^{n-2}\frac{1}{h+1}a_ha_{n-h-1}&=4b_{n-1}-a_n-3a_{n-1}\\
  2\sum\limits_{h=0}^{n-1}\frac{1}{h+1}a_ha_{n-h-1}&=3b_{n-1}-a_n\\
 \end{align*}

The first equation in the proposition is proved. The approach is similar for the second one. Using Equation
(\ref{eq:bnmixte}) and Equation (\ref{eq:anfctbn}) we deduce:

 \begin{align*}
 b_n-3b_{n-1}-a_{n-1}&=\sum\limits_{h=1}^{n-2}(a_h+b_h)b_{n-h-1}=2\sum\limits_{h=1}^{n-2}(2-\frac{1}{h+1})a_hb_{n-h-1}\\
 &=4\sum\limits_{h=1}^{n-2}a_hb_{n-h-1}-2\sum\limits_{h=1}^{n-2}\frac{1}{h+1}a_hb_{n-h-1}\\
 &=4(a_n-b_{n-1}-a_{n-1})-2\sum\limits_{h=1}^{n-2}\frac{1}{h+1}a_hb_{n-h-1} 
 \end{align*}
 
This implies:
\begin{align*}
  2\sum\limits_{h=1}^{n-2}\frac{1}{h+1}a_hb_{n-h-1}&=4a_n-b_n-b_{n-1}-3a_{n-1}\\
  2\sum\limits_{h=1}^{n-2}\frac{1}{3h+1}b_hb_{n-h-1}&=4a_n-b_n-b_{n-1}-3a_{n-1}\\
  2\sum\limits_{h=1}^{n-1}\frac{1}{3h+1}b_hb_{n-h-1}&=4a_n-b_n-b_{n-1}-3a_{n-1}+\frac{2}{3n-2}b_{n-1}\\
  2\sum\limits_{h=1}^{n-1}\frac{1}{3h+1}b_hb_{n-h-1}&=4a_n-b_n-b_{n-1}-3a_{n-1}+3a_{n-1}-b_{n-1}\\
  2\sum\limits_{h=1}^{n-1}\frac{1}{3h+1}b_hb_{n-h-1}&=4a_n-b_n-2b_{n-1}\\
  2\sum\limits_{h=0}^{n-1}\frac{1}{3h+1}b_hb_{n-h-1}&=4a_n-b_n\\
 \end{align*}

 The second equation in the conclusion is now proved.
 \end{proof}

\section{Threshold sequences and Motzkin-like paths}\label{sec:other}

We now show that threshold sequences may be seen as variants of Motzkin paths.

Defined in \cite{donaghey1977motzkin} in relation with the Motzkin numbers used in \cite{motzkin1948relations},
a {\em Motzkin path} of length $n$ is a path on the integral lattice $\mathbb{Z}\times \mathbb{Z}$ never going below the $x$-axis, 
starting in position $(0,0)$ and whose steps are of three types: up steps (1,1), 
horizontal steps (1,0) and down steps (1,-1). Here, step $(a,b)$ indicates that the next position in the path is reached
by moving $a$ units to right and $b$ units to top with respect to the current position. Motzkin paths enumerate various combinatorial
objects, as shown in \cite{donaghey1977motzkin,stanley1999vol2} but also more recently in 
\cite{guibert2001vexillary,sulanke2001bijective,deutsch2002bijection,prodinger2020two}. Some generalizations of Motzkin paths have also
been investigated \cite{humphreys2010history}, among which those with long horizontal steps \cite{barcucci1999eco}.

{\em Motzkin $n$-paths} are Motzkin paths of length $n$ ending in $(n,0)$. They are counted by the formula below \cite{donaghey1977motzkin}:
\begin{equation}
M_n=\sum \limits_{k=0}^{\lfloor\frac{n}{2}\rfloor}\binom{n}{2k}C_k 
\label{eq:Motzkinnumber} 
\end{equation}

\noindent where $C_k$ is the Catalan number. 

Let us call a {\em $(k,l)$-extended Motzkin path} of length $n$ every path on the integral lattice $\mathbb{Z}\times \mathbb{Z}$ 
never going below the $x$-axis,  starting in position $(0,0)$ and ending in position $(n,h)$, with $0\leq h\leq l$, 
whose $n$ steps are of three types: up steps $(1,u)$ with $u\geq 1$, horizontal steps $(1,0)$ and down steps $(1,-t)$ with 
$1\leq t\leq k-1$. Similarly to Motzkin $n$-paths, we call a {\em $(k,l)$-extended Motzkin $n$-path} every 
$(k,l)$-extended Motzkin path with $n$ steps that starts in  $(0,0)$ and ends in $(n,0)$.

\begin{proposition}
 Let $k,l,n$ be three integers such that $k\geq 2$, $0\leq l\leq k-2$ and $n\geq 1$.  The 
 $(k,l)$-extended Motzkin paths of length $n$ and endpoint $(n,l)$ are in bijection with the proper $(k,l)$-threshold 
 sequences of length $n$.
 \label{prop:eqMotzkin}
\end{proposition}

\begin{proof}
 We associate with the proper $(k,l)$-threshold sequence $S$ of length $n$ the path $Path(S)$ starting in position $(0,0)$ and whose $i$-th
 step depends on the value of $s_i-s_{i-1}-k$, for $1\leq i\leq n$ ($s_0$ is supposed to be equal to $0$):
 
 \begin{enumerate}[itemsep=0pt]
  \item[$\bullet$] if $s_i-s_{i-1}>k$, then the $i$-th step is an up step $(1, s_i-s_{i-1}-k)$
  \item[$\bullet$] if $s_i-s_{i-1}=k$, then the $i$-th step is an horizontal step $(1, 0)$
  \item[$\bullet$] if $s_i-s_{i-1}<k$, then the $i$-th step is a down step $(1, s_i-s_{i-1}-k)$
 \end{enumerate}

Let $(i,y_i)$, $1\leq i\leq n$, be the point the
 path reaches after the $i$-th step. Then, according to the definition of $Path(S)$:
 
 \begin{equation}
 y_i=\sum \limits_{p=1}^{i}(s_p-s_{p-1}-k)=s_i-ik.
 \label{eq:xi}
 \end{equation}
 
 Since $S$ is a proper $(k,l)$-threshold sequence, we have that $y_i\geq 0$ and
 also that $y_n=s_n-nk= l$. Thus the endpoint of $Path(S)$ i $(n,l)$.  In order to show that $Path(S)$ is a $(k,l)$-extended Motzkin path of length $n$, it remains to
 verify  that the down steps satisfy $1\leq -(s_i-s_{i-1}-k)\leq k-1$. We have $s_i-s_{i-1}<k$ by the definition of a down step,
 and thus the inequality on the left side is verified. Moreover,  $-(s_i-s_{i-1}-k)\leq k-1$ is equivalent with
 $s_i-s_{i-1}\geq 1$ and this is necessarily true since threshold sequences are strictly increasing sequences. Thus the right side
 inequality is verified too.
 
 Equation (\ref{eq:xi}) easily implies that the function $Path()$ is an injective function. To show it is a bijection, consider
 a $(k,l)$-extended Motzkin path of length $n$ and endpoint $(n,l)$, and let $(i,y_i)$ be the coordinates of the point reached after $i$ steps. We have $y_i\geq 0$
 by definition. We also have
 $y_n=l$ and, since the down steps $(1,-t)$ satisfy $1\leq t\leq k-1$, we deduce that 
 $y_i\leq y_n+(n-i)(k-1)\leq l+nk-ik-n+i$, thus $y_i+ik \leq nk+l+(i-n)\leq nk+l$. Then, defining
 $S$ as the sequence with $s_i=y_i+ik$ for all $i$, $1\leq i\leq n$, we have $s_i\leq nk+l$ and $s_i-ik=y_i\geq 0$,
 thus $S$ is a $(k,l)$-threshold sequence. Moreover, $s_n-nk=y_n=l$ and the proof is finished.
\end{proof}

 \begin{figure}[t]
\vspace*{-1cm}
 \centering
 \includegraphics[width=0.3\textwidth]{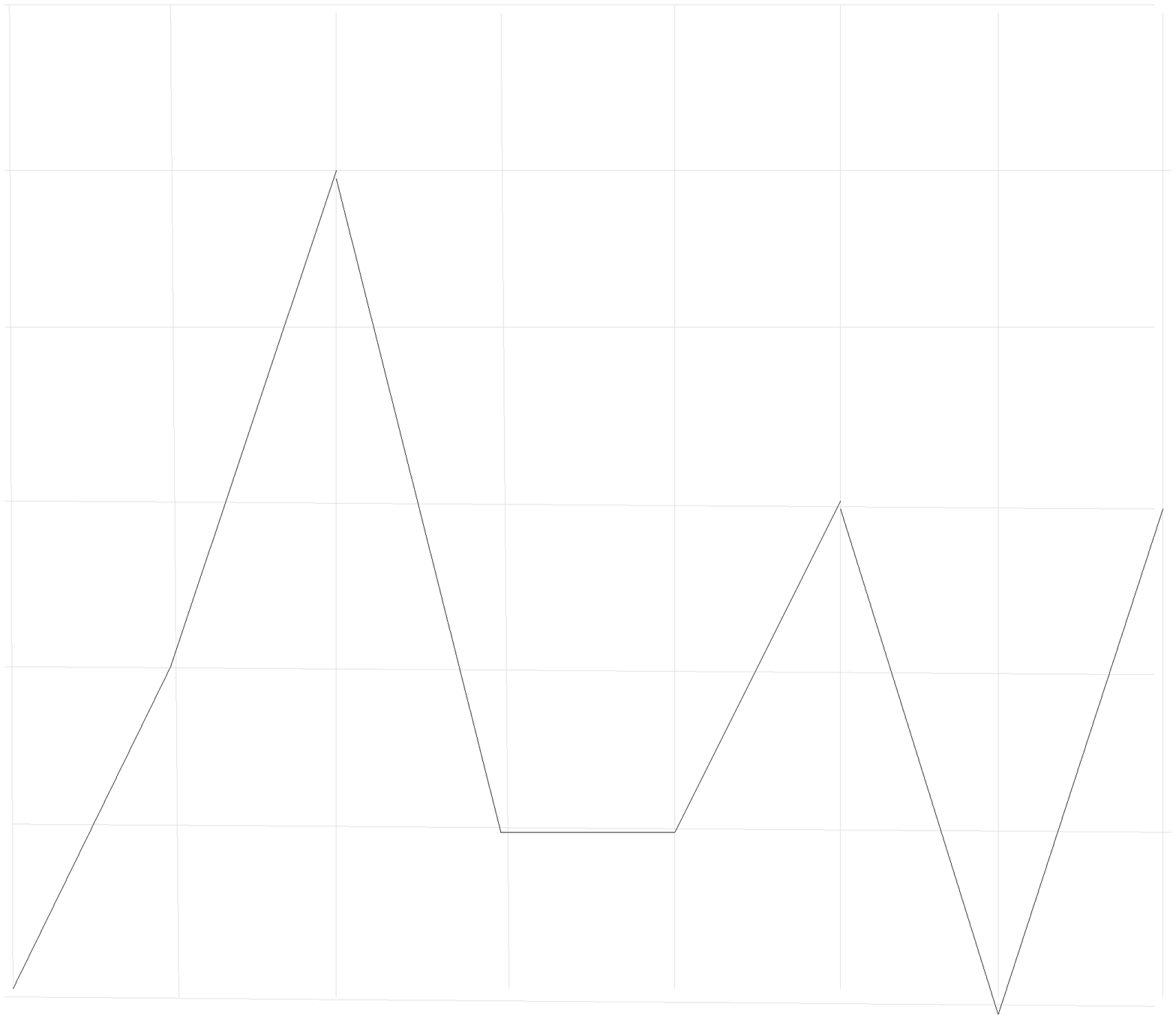}
 \caption{The $(5,4)$-extended Motzkin path associated with the $(5,4)$-threshold sequence  $S=(7\,\, 15\,\, 16\,\, 21\,\, 28\,\, 30\,\, 38)$.}
  \label{fig:MotzkinEx}
\end{figure}

\begin{example}
 Consider the proper $(5,3)$-threshold sequence $S$ of length $7$ given by $S=(7\,\, 15\,\, 16\,\, 21\,\, 28\,\, 30\,\,$ $38)$. The
 $(5,4)$-extended Motzkin path $Path(S)$ associated to it uses the following steps: $(1,2)$, $(1,3)$, $(1,-4)$, $(1,0)$, $(1,2)$, 
 $(1,-3)$, $(1,3)$. It starts in $(0,0)$ and ends in $(7,3)$. See Figure \ref{fig:MotzkinEx}. 
 \end{example}
 
Now, Proposition \ref{prop:eqMotzkin} and Corollary \ref{cor:properseq} imply:

\begin{corollary}
For each $l\geq 1$, the number of $(k,l)$-extended Motzkin paths of length $n$ ending in $(n,l)$ is equal to the Raney number:
$$R_{n-1}^{(k,k+l)}=\frac{k+l}{(k-1)(n-1)+k+l}\binom{kn+l-1}{n-1}.$$ 
\end{corollary}
 
Using Proposition   \ref{prop:eqMotzkin} for each $h=0, 1, \ldots, l$ and Corollary \ref{cor:numberthresholdseq}, we deduce that:

 \begin{corollary}
 The $(k,l)$-extended Motzkin paths of length $n$ 
 are in bijection with the $(k,l)$-threshold sequences of length $n$. Thus they are counted by the Raney number: 
 
 $$R_n^{(k,l+1)}=\frac{l+1}{kn+l+1}\binom{kn+l+1}{n} $$
 \label{cor:bijallMotzkin}
 \end{corollary}
 
 When $l=0$, the $(k,0)$-extended Motzkin paths of length $n$ are exactly the $(k,0)$-extended Motzkin $n$-paths. The case
 $k=2$ is worth to be noticed, since the down steps are in this case $(1,-1)$ steps only, as in the Motzkin paths. We deduce
 from the previous corollary and the remark that $R_n^{(2,1)}=C_n$ that:
 
\begin{corollary}
 Motzkin-like $n$-paths obtained by allowing arbitrarily long up steps in Motzkin $n$-paths are enumerated by the Catalan numbers:
 
$$C_n=\frac{1}{n+1}\binom{2n}{n}$$
 
\end{corollary}

%

\begin{figure}[t!]
\vspace*{-3.5cm}

 \centering
 \includegraphics[width=0.99\textwidth]{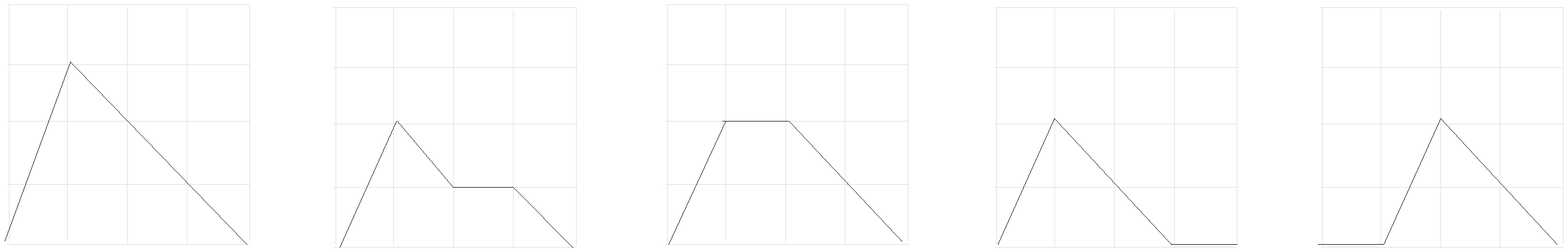}
\vspace*{-3.5cm}
\caption{The five Motzkin-like $4$-paths with long up steps, that are not Motzkin $4$-paths.} 
 \label{fig:Motzkin2}
\end{figure}

\begin{example}
 Figure \ref{fig:Motzkin2} shows the five Motzkin-like $4$-paths with arbitrarily long up steps which are not Motzkin $4$-paths.
 The number of Motzkin $4$-paths, computed with Equation~(\ref{eq:Motzkinnumber}) is 9. The total number of Motzkin-like $4$-paths 
 with arbitrarily long up steps is thus 14, which is the Catalan number $C_4$.
\end{example}

\section{Conclusion}\label{sec:conclusion}

The bijections between $(k,l)$-threshold sequences, ordered $(l+1)$-tuples of $k$-ary trees and Motzkin-like paths with 
long up and down steps we presented in the paper 
provide new combinatorial interpretations  for the Raney numbers, when the parameters $r$ and $k$  satisfy $r\leq k-1$. The
case where $k\leq r\leq 2k-2$ also gets new interpretations, since in this case the Raney numbers count proper
$(k,l)$-threshold sequences and Motzkin-like paths with long up and down steps that have a fixed endpoint. 

Threshold sequences may also be represented as particular cases of other combinatorial objects. For instance, $(k,l)$-threshold sequences of length
$n$ are also in bijection with
$k$-ballot sequences \cite{renault2007four} over the alphabet $\{A,B\}$ with  $a=kn+l+1$ letters $A$, $b=n$ letters $B$, whose letters 
$B$ are isolated and whose last letter is $B$. This is done by associating with each $(k,l)$-threshold sequence $S$ of length $n$
the $k$-ballot sequence  $W(S)=AW_1W_2\ldots W_n,$ where $W_i$ is a  sequence of $s_i-s_{i-1}$  letters $A$, followed by a 
letter $B$ ($s_0=0$ by convention). Then the number of $k$-ballot sequences with  $kb<a<kb+k$, 
whose letters $B$ are isolated and whose last letter is $B$, is equal to:
 
 $$R_b^{(k,a-kb)}=\frac{a-kb}{a}\binom{a}{b}$$
 
\noindent To see this, it is sufficient to define $l=a-kb-1$ and to conclude using Corollary \ref{cor:numberthresholdseq} and Equation (\ref{eq:Raney}). 

In consequence, threshold sequences are closely related to existing, and useful, combinatorial objects, and show efficient in bringing a new
point of view on these objects.

\bibliographystyle{plain}
\bibliography{SteepSeq}

\end{document}